\documentclass[review]{elsarticle}

\usepackage{mathptmx}
\usepackage{berasans}
\usepackage{beramono}

\usepackage[T1]{fontenc}
\usepackage[latin9]{inputenc}
\usepackage{geometry}
\usepackage{color}
\usepackage[english]{babel}
\usepackage{float}
\usepackage{amsthm}
\usepackage{amsmath}
\usepackage{amssymb}
\usepackage{graphicx}
\usepackage{esint}
\usepackage{natbib}
\usepackage{threeparttable}

\makeatletter

  \theoremstyle{definition}
  
  \theoremstyle{plain}
  
  \theoremstyle{definition}
  \newtheorem{example}{\protect\examplename}[section]
 \ifx\proof\undefined\
   \newenvironment{proof}[1][\proofname]{\par
     \normalfont\topsep6\p@\@plus6\p@\relax
     \trivlist
     \itemindent\parindent
     \item[\hskip\labelsep
           \scshape
       #1]\ignorespaces
   }{%
     \endtrivlist\@endpefalse
   }
   \providecommand{\proofname}{Proof}
 \fi
  \theoremstyle{plain}
  \newtheorem{lem}{\protect\lemmaname}[section]
  \theoremstyle{remark}
  \newtheorem{rem}{\protect\remarkname}[section]
  \theoremstyle{plain}
  \newtheorem{thm}{\protect\theoremname}[section]
  \theoremstyle{plain}
  \newtheorem{cor}{\protect\corollaryname}[section]

\usepackage{indentfirst}

\makeatother

\providecommand{\corollaryname}{Corollary}
\providecommand{\definitionname}{Definition}
\providecommand{\examplename}{Example}
\providecommand{\lemmaname}{Lemma}
\providecommand{\propositionname}{Proposition}
\providecommand{\remarkname}{Remark}
\providecommand{\theoremname}{Theorem}

\newenvironment{proofofof}[1]{\noindent{\bf Proof of Theorem #1. }}{ \qed }

\journal{Statistics and Probability Letters}

\begin{document}

\begin{frontmatter}

\title{Decomposing Correlated Random Walks on Common and Counter Movements}


\author[mymainaddress]{Tianyao Chen}

\author[mymainaddress]{Xue Cheng\corref{mycorrespondingauthor}}
\cortext[mycorrespondingauthor]{Corresponding author}
\ead{chengxue@pku.edu.cn}
\author[mymainaddress]{Jingping Yang}

\address[mymainaddress]{LMEQF, Department of Financial Mathematics, School of Mathematical Scienses, Peking University, Beijing 100871, China.}


\begin{abstract}
Random walk is one of the most classical and well-studied model in probability theory. For two correlated random walks on lattice, every step of the random walks has only two states, moving in the same direction or moving in the opposite direction. This paper presents a decomposition method to study the dependency structure of the two correlated random walks. By applying change-of-time technique used in continuous time martingales (see for example \cite{karatzas2012brownian} for more details), the random walks are decomposed into the composition of two independent random walks $X$ and $Y$ with change-of-time $T$, where $X$ and $Y$ model the common movements and the counter movements of the correlated random walks respectively. Moreover, we give a sufficient and necessary condition for mutual independence of $X$, $Y$ and $T$.
\end{abstract}

\begin{keyword}
correlated random walks, change-of-time, common movement, counter movement
\end{keyword}

\end{frontmatter}


\section{Introduction}\label{introduction}


Temporal correlated random walks have been widely considered. \cite{gillis1955correlated} and \cite{renshaw1981correlated} studied random walks on a $d$-dimensional lattice such that, at each step, the distribution depends on the state of previous step. \cite{chen1994general} considered a general correlated random walk as a Markov chain which includes a large number of examples as special cases. On the other hand, behavior between two simple independent random walks on graph were studied in \cite{krishnapur2004recurrent} and \cite{barlow2012collisions}. Given two independent random walks $S^1$ and $S^2$, \cite{ackermann2004independence} searched for stopping times $\tau$ such that $S^1_{\tau}$ and $S^2_{\tau}$ are independent. However, as far as we know, spatial correlation between two random walks have not been studied yet.

Consider two random walks, $\{B_n,n\geq1\}$ and $\{W_n,n\geq1\}$, on lattice. Let $(\xi_n,\eta_n)=(B_n-B_{n-1},W_n-W_{n-1})$, then $\xi_n,\eta_n\in\{1,-1\}$ satisfy that
$$\xi_n=\eta_n\quad or \quad\xi_n=-\eta_n,$$
i.e., there are two possible movements for every step of $(B,W)$, common movement or counter movement. Define
$$Q_n=\begin{cases}1,&\quad\mbox{if }\xi_n=\eta_n, \\0,&\quad\mbox{if }\xi_n=-\eta_n,\end{cases}$$
which can be considered as a state process specifying the common and counter movements of $B$ and $W$.

First we introduce some notations. For each $n\ge1$: \begin{description}
             \item[$T_n\triangleq\sum_{k=1}^nQ_k$,] number of common movements till step $n$; similarly we can define the number of counter movements $S_n\triangleq\sum_{k=1}^n(1-Q_k)=n-T_n$.
             \item[$\alpha_n\triangleq\inf\{k:T_k=n\}$,] total number of steps when $B$ and $W$ have got $n$ common moves, where we define $\inf\emptyset=\infty$.
Similarly we can define $\beta_n\triangleq\inf\{k:S_k=n\}$. 
             \item[$X_n\triangleq\sum_{k=1}^{\alpha_n}\xi_kQ_k$,] sum of the first $n$ common movements when $\alpha_n<\infty$; similarly we can define $Y_n\triangleq\sum_{k=1}^{\beta_n}\xi_k(1-Q_k)$ when $\beta_n<\infty$.
           \end{description}

According to these definitions, it is easy to obtain what we shall call the \emph{common decomposition of random walks} $B\ \mbox{and}\ W$,
\begin{equation}\label{common-decom}
  B_n=X_{T_n}+Y_{S_n},\ W_n=X_{T_n}-Y_{S_n}.
\end{equation}

The following example shows a scenario of relationship among these processes.
\begin{example} A sample path of $B_n$ and $W_n$ is shown in Figure \ref{BWXY}. The values of $B_n$, $W_n$,$T_n$,$S_n$,$\alpha_n$,$\beta_n$,$X_n$ and $Y_n$ are given in Table \ref{value}. The blank of the table means that we can not get the value from Figure \ref{BWXY}.
\begin{figure}[htbp]
\centering
\includegraphics[width=15cm,height=6cm]{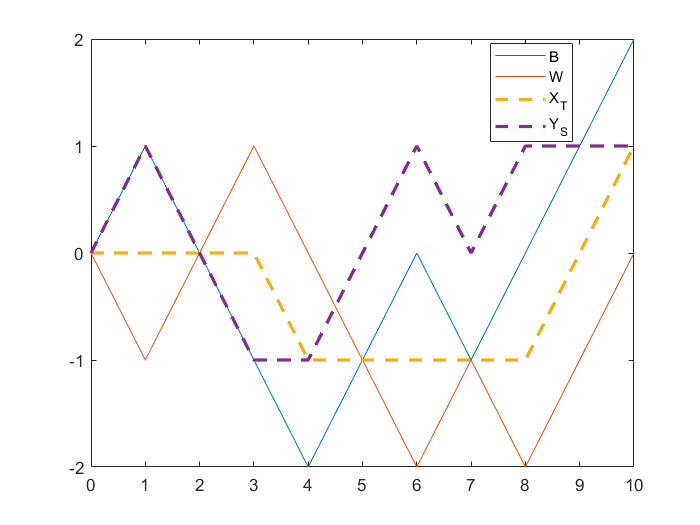}
\caption{A sample of $B_n$,$W_n$,$X_{T_n}$ and $Y_{S_n}$}
\label{BWXY}
\end{figure}
\begin{table}[htbp]
\centering
\small
\caption{A sample of underlying stochastic processes}
\label{value}
\begin{tabular}{ccccccccccc}
\hline
          &$n=1$&$n=2$&$n=3$&$n=4$&$n=5$&$n=6$&$n=7$&$n=8$&$n=9$&$n=10$\\
\hline
$B_n$     &1    &0    &-1   &-2   &-1   &0    &-1   &0    &1    &2  \\
$W_n$     &-1   &0    &1    &0    &-1   &-2   &-1   &-2   &-1   &0  \\
\hline
$T_n$     &0    &0    &0    &1    &1    &1    &1    &1    &2    &3    \\
$S_n$     &1    &2    &3    &3    &4    &5    &6    &7    &7    &7\\
$\alpha_n$&4    &9    &10   &     &     &     &     &     &     &\\
$\beta_n$ &1    &2    &3    &5    &6    &7    &8    &     &     &     \\
$X_n$     &-1   &0    &1    &     &     &     &     &     &     &\\
$Y_n$     &1    &0    &-1   &0    &1    &0    &1    &     &     &\\
\hline
\end{tabular}
\begin{tablenotes}
\item \hspace*{2.6cm}\emph{Blanks of table represent we need further information of $B_n$ and $W_n$ to confirm the values.}
\end{tablenotes}
\end{table}
\end{example}

The decomposition (\ref{common-decom}) states that the dependency structure of $B$ and $W$ can be described by the processes $\{X_n,n\geq 1\},\{Y_n,n\geq 1\}$ and $\{T_n,n\geq 1\}$. In this paper, we will consider some properties of the three random processes, especially the independence of the processes. Note that the decomposition \eqref{common-decom} can be regarded as a discrete-time regime switching model. Since random walk is one of the most well-studied and well-applied topics in probability theory (see \cite{spitzer2013principles},\cite{lawler2010random}), this decomposition may have applications in various subjects. In finance, a random walk can be used as the sign process of some special sequences, like the discrete asset price sequence (after multiplying the volatility) in Bachelier model \cite{bachelier1900theorie}, thus our common  decomposition method can be used to study the trend information contained in the sequence.


This paper is organized as follows. In Section 2, we set up the model concretely and discuss the independence property of $X$,$Y$ and $T$. A specific example picked from finance is also considered in this section to show applications of our method. In Section 3, we give a conclusion.

\section{Main Results}\label{regime switch and random walk}

\subsection{Model Setup}
Consider a filtered probability space in discrete time $\left(\Omega, \mathcal F, \mathbb F=\{\mathcal F_n\}_{n\in \mathbb Z}, P\right)$,
where $\mathcal F_0=\{\Omega,\emptyset\}$. As introduced before, for two standard random walks $\{B_n,n\geq 1\}$ and $\{W_n,n\geq 1\}$ with respect to $\mathbb{F}$, we write $\{\xi_n\}$ and $\{\eta_n\}$ as their increments respectively,
$$\xi_n=B_n-B_{n-1},\ \eta_n =W_n-W_{n-1},$$
where $B_0=W_0=0$. We assume that for each $n\geq 1$,
\begin{equation}\label{self-con-proba}
  P(\xi_n=1|\mathcal{F}_{n-1})=P(\xi_n=-1|\mathcal{F}_{n-1})
  =P(\eta_n=1|\mathcal{F}_{n-1})=P(\eta_n=-1|\mathcal{F}_{n-1})=\frac12.
\end{equation}
Thus, $\xi_n$ is independent of  $\mathcal{F}_{n-1}$, so is $\eta_n$. Namely, the independent-increments property of $B$ and $W$ are still hold, so we restrict our attention to the correlation of $\xi$ and $\eta$, but not to the autocorrelation of each sequence. However, we can not say that $(\xi_n,\eta_n)$ is independent of $\mathcal{F}_{n-1}$. One aim of this article is to study the property of $(\xi_n,\eta_n)$.

Let $\vartheta_n$, a $\mathcal F_{n-1}$ measurable random variable, denote the probability that $B$ and $W$ increase together in the $n$th step conditional on the information till $n-1$, i.e., $\label{increment probability1}\vartheta_n=P(\xi_n=1,\eta_n=1|\mathcal{F}_{n-1}).$
Then the distribution of $\{\xi_j, \eta_j, 1\leq j\leq n\}$ is given accordingly. 

By immediate calculation, we have
\begin{equation}\label{increment probability2} P(\xi_n=-1,\eta_n=1|\mathcal{F}_{n-1})=\frac12-\vartheta_n,P(\xi_n=1,\eta_n=-1|\mathcal{F}_{n-1})=\frac12-\vartheta_n,\end{equation}
and
\begin{equation}\label{increment probability3}P(\xi_n=-1,\eta_n=-1|\mathcal{F}_{n-1})=\vartheta_n.\end{equation}


%
We first investigate the properties of $X$ and $Y$. Note that when $\alpha_n<\infty$ and $\beta_n<\infty$,
$$X_n=\sum_{k=1}^{\alpha_n}\xi_kQ_k=\sum_{i=1}^{n}\xi_{\alpha_i},\quad Y_n=\sum_{k=1}^{\beta_n}\xi_k(1-Q_k)=\sum_{i=1}^{n}\xi_{\beta_i}.$$ However, when $\alpha_n=\infty$ or $\beta_m=\infty$, neither $\xi_{\alpha_n},\xi_{\beta_m}$ nor $X_n,Y_n$ has been well-defined. Therefore, proper adjustments are needed. Thus we introduce two i.i.d. sequences $\{\zeta_n\}$ and $\{\psi_n\}$ with  $P(\zeta_n=1)=P(\zeta_n=-1)=P(\psi_n=1)=P(\psi_n=-1)=\frac12,n\geq 1$, and assume that $\{\zeta_n\}$,$\{\psi_n\}$ and $\mathcal{F}_{\infty}$ are mutually independent. We modify $\xi_{\alpha_n}$ and $\xi_{\beta_m}$ as
\begin{equation}\label{def tilde xi}
 {\tilde\xi_{\alpha_n}=}
\begin{cases}
\xi_i,&\mbox{if } \alpha_n=i\\
\zeta_n,&\mbox{if } \alpha_n=\infty
\end{cases} ,\ \
\tilde\xi_{\beta_m}=
\begin{cases}\xi_i, &\mbox{if } \beta_m=i\\\psi_m, &\mbox{if } \beta_m=\infty\end{cases}.
\end{equation}

For the cases where $\alpha_n$ and $\beta_m$ could be infinite, a natural complementary definition of $X_n$ and $Y_n$ is
$$X_n\triangleq\sum_{i=1}^{n}\tilde\xi_{\alpha_i},\quad Y_n\triangleq\sum_{i=1}^{n}\tilde\xi_{\beta_i}.$$
It is notable $\xi_{\alpha_n}$ (resp. $\xi_{\beta_n}$) represents the $n$th common (resp. counter) movements of $B$ and $W$. But when $\alpha_n=\infty$ (resp. $\beta_n=\infty$), we have $T_i<n$ for any $i\geq1$ (resp. $S_i<n$,for any $i\geq1$), in other words, there are no more than $n$ common (resp. counter) movements during the whole time period. Hence, modifying the definition of $\xi_{\alpha_n}$ (resp. $\xi_{\beta_n}$) when $\alpha_n=\infty$ (resp. $\beta_n=\infty$) will not affect the common decomposition \eqref{common-decom} and $X_n$ (resp. $Y_n$) still characterize the common (resp. counter) movements of $B$ and $W$.


\subsection{Independency Property of $X$ and $Y$}
We can now formulate our first main result.

\begin{thm}\label{independent random walks}Suppose that \eqref{self-con-proba} holds. For $\{\tilde{\xi}_{\alpha_n},n\geq1\}$,$\{\tilde\xi_{\beta_n},n\geq1\}$ as defined in \eqref{def tilde xi}, $\{\tilde{\xi}_{\alpha_n},\tilde\xi_{\beta_n},n\geq1\}$ are i.i.d. random variables with the distribution $P(\tilde\xi_{\alpha_n}=1)=P(\tilde\xi_{\alpha_n}=-1)=P(\tilde\xi_{\beta_n}=1)=P(\tilde\xi_{\beta_n}=-1)=\frac12$. As a special case, if
\begin{equation}\label{limit-assumption}
 P(\lim_{n\to\infty}T_n=\infty)=P(\lim_{n\to\infty}S_n=\infty)=1,
\end{equation}
then $\tilde\xi_{\alpha_n}=\xi_{\alpha_n}, \tilde\xi_{\beta_n}=\xi_{\beta_n}, n\geq 1$, and $\{\xi_{\alpha_n},\xi_{\beta_n},n\geq1\}$ are i.i.d. random variables.
\end{thm}
%

Before proving the theorem, we first state some lemmas.

\begin{lem}\label{property of alpha&beta}Suppose that \eqref{self-con-proba} holds. We have
\begin{description}
\item[1)] For $n,m\geq 1$, $\alpha_n\geq n$ and $\beta_m\geq m$;
\item[2)] Given $n_2>n_1\geq1$, if $\alpha_{n_1}<\infty$(resp. $\beta_{n_1}<\infty$), then $\alpha_{n_1}<\alpha_{n_2}$(resp. $\beta_{n_1}<\beta_{n_2}$). On the contrary, for any $\alpha_{n_1}<\alpha_{n_2}$(resp. $\beta_{n_1}<\beta_{n_2}$), we must have $n_1<n_2$. Hence both $\{\alpha_n\}$ and $\{\beta_n\}$ are increasing sequences of stopping times.
\item[3)] For each fixed $n,m\geq 1$, $\{\alpha_n=\beta_m<\infty\}\bigcup\{\alpha_n=\infty,\beta_m=\infty\}=\emptyset.\label{infty}$

\end{description}
\end{lem}
\begin{proof}
\begin{description}
The proofs of 1) and 2) are straightforward by definition, here we only give the proof of 3).

Suppose that $\alpha_n=\beta_m=k<\infty.$ By definition
$$\{\alpha_n=k\}=\{T_{k-1}=n-1,\xi_k=\eta_k\},\{\beta_m=k\}=\{S_{k-1}=m-1,\xi_k=-\eta_k\},$$
thus $\{\alpha_n=k\}\bigcap\{\beta_m=k\}=\emptyset$, contrary to the assumption $\alpha_n=\beta_m=k$.

If $\alpha_n=\infty$ and $\beta_m=\infty,$ then $\inf\{k:T_k=n\}=\inf\{k:S_k=m\}=\infty$. Hence $T_k\leq n,S_k\leq m, k\in\mathbb{N}$, which is contrary to that $T_k+S_k=k$ when $k>m+n$.


\end{description}
\end{proof}

The following lemma provides the distributions of $\tilde \xi_{\alpha_n}$ and $\tilde\xi_{\beta_m}$, $n\geq1,m\geq1$.

\begin{lem}
\label{BIN}
Under the assumptions of Theorem \ref{independent random walks}, for any $n\geq1,m\geq1$, we have
\begin{equation}\label{distr-single}
 P(\tilde \xi_{\alpha_n}=1)=P(\tilde \xi_{\alpha_n}=-1)=P(\tilde \xi_{\beta_m}=1)=P(\tilde \xi_{\beta_m}=-1)=\frac12.
\end{equation}
\end{lem}
\begin{proof}
 For simplicity, we only consider $\tilde\xi_{\alpha_n}$. The proof for $\tilde\xi_{\beta_m}$ is similar.
Note that
$$P(\tilde \xi_{\alpha_n}=1)=P(\xi_{\alpha_n}=1,\alpha_n<\infty)+P(\xi_{\alpha_n}=1,\alpha_n=\infty).$$
For the first term in the righthand of the above equation,
\begin{equation}\label{distr-noninfi}
\begin{split}
P(\xi_{\alpha_n}=1,\alpha_n<\infty)
=&\sum_{i\geq n, i<\infty}P(\xi_{i}=1,\alpha_n=i)\\
=&\sum_{i\geq n, i<\infty}P(\xi_{i}=\eta_i=1,T_{i-1}=n-1)\\
=&\sum_{i\geq n, i<\infty}E\left[1_{T_{i-1}=n-1}P(\xi_{i}=\eta_i=1|\mathcal{F}_{i-1})\right]\\
=&\sum_{i\geq n, i<\infty}E\left[1_{T_{i-1}=n-1}P(\xi_{i}=\eta_i=-1|\mathcal{F}_{i-1})\right]\\
=&P(\xi_{\alpha_n}=-1,\alpha_n<\infty),
\end{split}
\end{equation}
thus we can get $P(\tilde\xi_{\alpha_n}=1,\alpha_n<\infty)=\frac12 P(\alpha_n<\infty)$.
On the other hand,
 $\{\alpha_n=\infty\}\in\mathcal{F}_{\infty}$ and hence the set is independent from $\zeta_n,$
\begin{equation}\label{alpha_n=infty}
P(\tilde\xi_{\alpha_n}=1,\alpha_n=\infty)=P(\zeta_n=1,\alpha_n=\infty)
=\frac12P(\alpha_n=\infty).
\end{equation}
Then, $P(\tilde\xi_{\alpha_n}=1)=P(\tilde\xi_{\alpha_n}=-1)=
\frac12\big(P(\alpha_n=\infty)+P(\alpha_n<\infty)\big)=\frac12$ as claimed.

\end{proof}

\begin{lem}\label{two parts} Under the assumptions of Theorem \ref{independent random walks}, for $n_k>n_{k-1}>\cdots>n_1,m_l>m_{l-1}>\cdots>m_1$, we have
\begin{equation}\label{A-B}
\begin{split}
  P(\tilde\xi_{\alpha_{n_i}}=x_i,\tilde\xi_{\beta_{m_j}}=y_j,1\leq i\leq k,1\leq j\leq l)=&\frac12 P(\tilde\xi_{\alpha_{n_i}}=x_i,\tilde\xi_{\beta_{m_j}}=y_j,1\leq i\leq k-1,1\leq j\leq l,\alpha_{n_k}>\beta_{m_l})\\
  &+\frac12 P(\tilde\xi_{\alpha_{n_i}}=x_i,\tilde\xi_{\beta_{m_j}}=y_j,1\leq i\leq k,1\leq j\leq l-1,\alpha_{n_k}<\beta_{m_l}).
\end{split}
\end{equation}
\end{lem}
\begin{proof}

 Recall the conclusion in Lemma \ref{property of alpha&beta}, we can divide $\Omega$ into two disjoint sets $\{\alpha_{n_k}>\beta_{m_l}\}$ and $\{\alpha_{n_k}<\beta_{m_l}\}$. Then
\begin{equation}\label{divide}
\begin{split}
  &P(\tilde\xi_{\alpha_{n_i}}=x_i,\tilde\xi_{\beta_{m_j}}=y_j,1\leq i\leq k,1\leq j\leq l)\\
  =&P(\tilde\xi_{\alpha_{n_i}}=x_i,\tilde\xi_{\beta_{m_j}}=y_j,1\leq i\leq k,1\leq j\leq l,\alpha_{n_k}>\beta_{m_l})+P(\tilde\xi_{\alpha_{n_i}}=x_i,\tilde\xi_{\beta_{m_j}}=y_j,1\leq i\leq k,1\leq j\leq l,\alpha_{n_k}<\beta_{m_l}).
\end{split}
\end{equation}
For the first part in the righthand of \eqref{divide}, we apply similar techniques as in (\ref{distr-noninfi}) and (\ref{alpha_n=infty}),
\begin{equation}
\label{x_k to -x_k}
\begin{split}
&P(\tilde\xi_{\alpha_{n_i}}=x_i,\tilde\xi_{\beta_{m_j}}=y_j,1\leq i\leq k-1,1\leq j\leq l,\tilde\xi_{\alpha_{n_k}}=x_{k},\alpha_{n_k}>\beta_{m_l})\\
=&\sum_{n_k\leq h<\infty}E\left[P\left(\xi_{\alpha_{n_i}}=x_i,\xi_{\beta_{m_j}}=y_j,1\leq i\leq k-1,1\leq j\leq l,\tilde\xi_{\alpha_{n_k}}=x_{k},h=\alpha_{n_k}>\beta_{m_l}|\mathcal F_{h-1}\right)\right]\\
&+P(\zeta_{n_k}=x_k)P(\tilde\xi_{\alpha_{n_i}}=x_i,\tilde\xi_{\beta_{m_j}}=y_j,1\leq i\leq k-1,1\leq j\leq l, \infty=\alpha_{n_k}>\beta_{m_l})\\
=&\sum_{n_k\leq h<\infty}E\left[1_{\{\xi_{\alpha_{n_i}}=x_i,\xi_{\beta_{m_j}}=y_j,1\leq i\leq k-1,1\leq j\leq l;T_{h-1}=n_k-1\}}P\left(\xi_{h}=\eta_h=x_{k}|\mathcal F_{h-1}\right)\right]\\
&+P(\zeta_{n_k}=-x_{k})P(\tilde\xi_{\alpha_{n_i}}=x_i,\tilde\xi_{\beta_{m_j}}=y_j,1\leq i\leq k-1,1\leq j\leq l, \infty=\alpha_{n_k}>\beta_{m_l})\\
=&\sum_{n_k\leq h<\infty}E\left[1_{\{\xi_{\alpha_{n_i}}=x_i,\xi_{\beta_{m_j}}=y_j,1\leq i\leq k-1,1\leq j\leq l;T_{h-1}=n_k-1\}}P\left(\xi_{h}=\eta_h=-x_{k}|\mathcal F_{h-1}\right)\right]\\
&+P(\zeta_{n_k}=-x_{k})P(\tilde\xi_{\alpha_{n_i}}=x_i,\tilde\xi_{\beta_{m_j}}=y_j,1\leq i\leq k-1,1\leq j\leq l, \infty=\alpha_{n_k}>\beta_{m_l})\\
=&P(\tilde\xi_{\alpha_{n_i}}=x_i,\tilde\xi_{\beta_{m_j}}=y_j,1\leq i\leq k-1,1\leq j\leq l,\tilde\xi_{\alpha_{n_k}}=-x_{k},\alpha_{n_k}>\beta_{m_l}),
\end{split}
\end{equation}
where the third equality can be obtained by $P(\xi_h=\eta_h=x_k|\mathcal{F}_{h-1})=P(\xi_h=\eta_h=-x_k|\mathcal{F}_{h-1})$.
 Note that
\begin{eqnarray*}
&&P(\tilde\xi_{\alpha_{n_i}}=x_i,\tilde\xi_{\beta_{m_j}}=y_j,1\leq i\leq k-1,1\leq j\leq l,\alpha_{n_k}>\beta_{m_l})\\
&=&P(\tilde\xi_{\alpha_{n_i}}=x_i,\tilde\xi_{\beta_{m_j}}=y_j,1\leq i\leq k-1,1\leq j\leq l,\tilde\xi_{\alpha_{n_k}}=x_k,\alpha_{n_k}>\beta_{m_l})\\
&&+P(\tilde\xi_{\alpha_{n_i}}=x_i,\tilde\xi_{\beta_{m_j}}=y_j,1\leq i\leq k-1,1\leq j\leq l,\tilde\xi_{\alpha_{n_k}}=-x_k,\alpha_{n_k}>\beta_{m_l}),
\end{eqnarray*}
then from \eqref{x_k to -x_k} follows that
\begin{equation}\label{alpha>beta}
\begin{split}
P(\tilde\xi_{\alpha_{n_i}}=x_i,\tilde\xi_{\beta_{m_j}}=y_j,1\leq i\leq k,1\leq j\leq l,\alpha_{n_k}>\beta_{m_l})=\frac12 P(\tilde\xi_{\alpha_{n_i}}=x_i,\tilde\xi_{\beta_{m_j}}=y_j,1\leq i\leq k-1,1\leq j\leq l,\alpha_{n_k}>\beta_{m_l}).
\end{split}
\end{equation}
By similar arguments we can get
\begin{equation}\label{alpha<beta}
\begin{split}
P(\tilde\xi_{\alpha_{n_i}}=x_i,\tilde\xi_{\beta_{m_j}}=y_j,1\leq i\leq k,1\leq j\leq l,\alpha_{n_k}<\beta_{m_l})=\frac12 P(\tilde\xi_{\alpha_{n_i}}=x_i,\tilde\xi_{\beta_{m_j}}=y_j,1\leq i\leq k,1\leq j\leq l-1,\alpha_{n_k}<\beta_{m_l}).
\end{split}
\end{equation}
Combining (\ref{alpha>beta}) and (\ref{alpha<beta}) and applying (\ref{divide})
, it yields \eqref{A-B}.

\end{proof}

Now we can finish the proof of Theorem \ref{independent random walks}.

\begin{proofofof}{\ref{independent random walks}}
 By Lemma \ref{BIN}, all the $\tilde\xi_{\alpha_n}$, $\tilde\xi_{\beta_n},n\geq1$ are identically distributed. We now proceed to show that they are independent.

It is sufficient to prove that for any $k,l\in \mathbb N,\{n_1,\cdots,n_k,m_1,\cdots,m_l\}\subset\mathbb N$ and $x_1, y_1, x_2, y_2$, $\cdots$, $x_k$, $y_l \in \{1,-1\}$,
\begin{equation}\label{inde-eq}
  P(\tilde\xi_{\alpha_{n_i}}=x_i,\tilde\xi_{\beta_{m_j}}=y_j,1\leq i\leq k,1\leq j\leq l)=\frac1{2^{k+l}}.
\end{equation}
Without loss of generality, we assume that $n_k>\cdots>n_1,m_l>\cdots>m_1$. 
If there are only $\alpha$'s, i.e., $l=0$, we put ${m_j}=0,1\leq j\leq l$ and define $\beta_0=0$, vice versa.
Set $M=k+l$, we will complete the proof by induction.

From Lemma \ref{BIN}, (\ref{inde-eq}) is true for $M=1$. Next, we assume (\ref{inde-eq}) is true for any $M<N$, and we will prove \eqref{inde-eq} is true for $M=N$.


It is evident that, if there are only $\alpha$'s, since $\beta_{m_l}=0$, we get $\{\alpha_{n_k}<\beta_{m_l}\}=\emptyset$. 
By (\ref{A-B}),
\begin{equation}
\label{only alpha}
\begin{split}
P(\tilde\xi_{\alpha_{n_i}}=x_i,1\leq i\leq k)
=&\frac12 P(\tilde\xi_{\alpha_{n_i}}=x_i,1\leq i\leq k-1,\alpha_{n_k}>\beta_{m_l})\\
=&\frac12 P(\tilde\xi_{\alpha_{n_i}}=x_i,1\leq i\leq k-1).
\end{split}
\end{equation}
Since \eqref{inde-eq} holds for any $M<N$, we have $P(\tilde\xi_{\alpha_{n_i}}=x_i,1\leq i\leq k)=\frac12\cdot\frac1{2^{N-1}}=\frac1{2^N}.$
 Thus (\ref{inde-eq}) is true for $M=N$. The case with only $\beta$s is similar.

 Next we consider the case that both $\alpha$'s and $\beta$s are contained. From our assumption,
\begin{equation}\label{1/2^(N-1)}
\begin{split}
&P(\tilde\xi_{\alpha_{n_i}}=x_i,\tilde\xi_{\beta_{m_j}}=y_j,1\leq i\leq k-1,1\leq j\leq l-1,\tilde\xi_{\alpha_{n_k}}=x_k,\tilde\xi_{\beta_{m_l}}=1)\\
&+P(\tilde\xi_{\alpha_{n_i}}=x_i,\tilde\xi_{\beta_{m_j}}=y_j,1\leq i\leq k-1,1\leq j\leq l-1,\tilde\xi_{\alpha_{n_k}}=x_k,\tilde\xi_{\beta_{m_l}}=-1)\\
=&P(\tilde\xi_{\alpha_{n_i}}=x_i,\tilde\xi_{\beta_{m_j}}=y_j,1\leq i\leq k-1,1\leq j\leq l-1,\tilde\xi_{\alpha_{n_k}}=x_k)=\frac1{2^{N-1}}.
\end{split}
\end{equation}
Applying Lemma \ref{two parts} for both of the two probabilities in the lefthand of \eqref{1/2^(N-1)},
\begin{equation}
\label{equation set1}
\begin{split}
&P(\tilde\xi_{\alpha_{n_i}}=x_i,\tilde\xi_{\beta_{m_j}}=y_j,1\leq i\leq k-1,1\leq j\leq l-1,\tilde\xi_{\alpha_{n_k}}=x_k,\alpha_{n_k}<\beta_{m_l})\\
&+\frac12 P(\tilde\xi_{\alpha_{n_i}}=x_i,\tilde\xi_{\beta_{m_j}}=y_j,1\leq i\leq k-1,1\leq j\leq l-1,\tilde\xi_{\beta_{m_l}}=1,\alpha_{n_k}>\beta_{m_l})\\
&+\frac12 P(\tilde\xi_{\alpha_{n_i}}=x_i,\tilde\xi_{\beta_{m_j}}=y_j,1\leq i\leq k-1,1\leq j\leq l-1,\tilde\xi_{\beta_{m_l}}=-1,\alpha_{n_k}>\beta_{m_l})=\frac1{2^{N-1}}.
\end{split}
\end{equation}
Note that in \eqref{equation set1} only the first term $P(\tilde\xi_{\alpha_{n_i}}=x_i,\tilde\xi_{\beta_{m_j}}=y_j,1\leq i\leq k-1,1\leq j\leq l-1,\tilde\xi_{\alpha_{n_k}}=x_k,\alpha_{n_k}<\beta_{m_l})$ contains $x_k$ and $x_k\in\{1,-1\}$, which implies that
\begin{equation}
\label{def a}
\begin{split}
&P(\tilde\xi_{\alpha_{n_i}}=x_i,\tilde\xi_{\beta_{m_j}}=y_j,1\leq i\leq k-1,1\leq j\leq l-1,\tilde\xi_{\alpha_{n_k}}=1,\alpha_{n_k}<\beta_{m_l})\\
=&P(\tilde\xi_{\alpha_{n_i}}=x_i,\tilde\xi_{\beta_{m_j}}=y_j,1\leq i\leq k-1,1\leq j\leq l-1,\tilde\xi_{\alpha_{n_k}}=-1,\alpha_{n_k}<\beta_{m_l}):=a.\\
\end{split}
\end{equation}
Similarly, consider $\frac1{2^{N-1}}=P(\tilde\xi_{\alpha_{n_i}}=x_i,\tilde\xi_{\beta_{m_j}}=y_j,1\leq i\leq k-1,1\leq j\leq l-1,\tilde\xi_{\beta_{m_l}}=y_l)$ and apply the same method as for (\ref{equation set1}) and (\ref{def a}), we have
\begin{equation}
\label{def b}
\begin{split}
&P(\tilde\xi_{\alpha_{n_i}}=x_i,\tilde\xi_{\beta_{m_j}}=y_j,1\leq i\leq k-1,1\leq j\leq l-1,\tilde\xi_{\beta_{m_l}}=1,\alpha_{n_k}>\beta_{m_l})\\
=&P(\tilde\xi_{\alpha_{n_i}}=x_i,\tilde\xi_{\beta_{m_j}}=y_j,1\leq i\leq k-1,1\leq j\leq l-1,\tilde\xi_{\beta_{m_l}}=-1,\alpha_{n_k}>\beta_{m_l}):=b.
\end{split}
\end{equation}
Then by Lemma \ref{two parts},
$$P(\tilde\xi_{\alpha_{n_i}}=x_i,\tilde\xi_{\beta_{m_j}}=y_j,1\leq i\leq k,1\leq j\leq l)=\frac12a+\frac12b$$
and from \eqref{1/2^(N-1)}, $\frac{1}{2^{N-1}}=a+b.$

Then it is immediate that (\ref{inde-eq}) is true for $M=N$. Thus independency of $\tilde\xi_{\alpha_n},\tilde\xi_{\beta_n},n\geq 1$ follows from induction.
\end{proofofof}


 By the definitions of $\{X_n,n\geq1\}$ and $\{Y_n,n\geq1\}$ and applying Theorem \ref{independent random walks}, we immediately have the following corollary.
\begin{cor}\label{corollary1}
$\{X_n,n\geq 1\}$ and $\{Y_n,n\geq 1\}$ are independent random walks.
\end{cor}

The two correlated standard random walks $B$ and $W$ can be decomposed as
\begin{eqnarray*}
B_n&=&\sum_{i=1}^{T_n}\xi_{\alpha_i}+\sum_{j=1}^{n-T_n}\xi_{\beta_j}=X_{T_n}+Y_{S_n},\\
W_n&=&\sum_{i=1}^{T_n}\xi_{\alpha_i}-\sum_{j=1}^{n-T_n}\xi_{\beta_j}=X_{T_n}-Y_{S_n}.
\end{eqnarray*}
Correlation between $B$ and $W$ can be illustrated by three processes, $X,Y$ and $T$, where $X$ and $Y$ indicates the same-direction and the opposite-direction moves of $B$ and $W$ respectively, and $T$ records the number of common movements. According to Theorem \ref{independent random walks} and Corollary \ref{corollary1}, when \eqref{self-con-proba} is satisfied, the distribution of $(X,Y)$ does not depend on the correlation of $B$ and $W$. Consequently, the dependency structure of $B$ and $W$ is only contained in $T$.

\begin{rem}
 Consider general random walks with $P(\xi_n=1|\mathcal{F}_{n-1})=P(\eta_n=1|\mathcal{F}_{n-1})=p,$ where $0<p<1,\ p\neq\frac12$. Set $\vartheta_n=P(\xi_n=\eta_n=1|\mathcal{F}_{n-1})$ as before, we still have $P(\xi_n=1,\eta_n=-1|\mathcal{F}_{n-1})=P(\xi_n=-1,\eta_n=1|\mathcal{F}_{n-1})$. Thus the proof for Theorem \ref{independent random walks} remains valid for $\{\tilde\xi_{\beta_1},\tilde\xi_{\beta_2},\cdots\}$. And then the counter-move process $Y$ is still a standard random walk.

 But for $X$, since $P(\xi_n=\eta_n=-1|\mathcal{F}_{n-1})=1-2p+\vartheta_n\neq P(\xi_n=\eta_n=1|\mathcal{F}_{n-1}),$ we only have $$P(\xi_{\alpha_n}=1,\alpha_n<\infty)=\frac12P(\alpha_n<\infty)-\frac{1-2p}{2}\sum_{i\geq n-1}P(T_i=n-1).$$

Under the assumption that $P(\lim_{i\to\infty}T_i=\infty)=1$, we can get
$$\sum_{i\geq n-1}P(T_i=n-1)>P(\bigcup_{i\geq n-1}\{T_i=n-1\})=P(\lim_{i\to\infty}T_i=\infty)=1,$$
and $\alpha_n<\infty$ a.s. for any $n\geq1$, thus $P(\xi_{\alpha_n}=1)=P(\xi_{\alpha_n}=1,\alpha_n<\infty)<p$ when $p<\frac12$ and $P(\xi_{\alpha_n}=1)>p$ when $p>\frac12$. It can be seen that if two correlated random walks have the same moving trend ($p\neq 1/2$), they can be decomposed as a random walk $X$ which have a even stronger trend and a white noise $Y$.

\end{rem}

Consider again the standard case $P(\xi_n=1|\mathcal{F}_{n-1})=P(\eta_n=1|\mathcal{F}_{n-1})=\frac12$. The ideal situation that the three process $X$, $Y$ and $T$ are independent may bring us a lot of convenience from both theoretic and practical views. We are thus present the following sufficient and necessary conditions for independence of the three processes.

We begin this with introducing a condition, which is commonly used in credit intensity models in math-finance (see \cite[Chapter 6]{bielecki2013credit} for more details):
\begin{description}
\item[C1)] For any $n\in \mathbb N$, the $\sigma-$fields $\mathcal H_n$ and $\mathcal{G}_{\infty}$ are conditionally independent given $\mathcal{G}_{n}$.
\end{description}
\begin{thm}\label{XYT independent}
 Let $\mathcal G_n\triangleq\sigma(T_k,k\le n)$, $\mathcal H_n\triangleq \sigma(B_k,W_k,k\le n)$. Then
 $\{X_n,n\geq 1\}$, $\{Y_n,n\geq 1\}$ and $\{T_n,n\geq 1\}$ are independent iff the condition C1) is satisfied.
 \end{thm}

Before prove Theorem \ref{XYT independent}, we first consider a lemma that will be used later.

\begin{lem}\label{lemma1-2}
If the condition C1) holds, then
$$P\left(\xi_n=\eta_n=1|\mathcal H_{n-1}\bigvee\mathcal G_n\right)=P\left(\xi_n=\eta_n=-1|\mathcal H_{n-1}\bigvee\mathcal G_n\right)$$ and $$P\left(\xi_n=-\eta_n=1|\mathcal H_{n-1}\bigvee\mathcal G_n\right)=P\left(\xi_n=-\eta_n=-1|\mathcal H_{n-1}\bigvee\mathcal G_n\right).$$
\end{lem}
\begin{proof}
For any $A\in \mathcal H_{n-1}\subseteq \mathcal F_{n-1}$, it holds that
\begin{equation}\label{condition=}
 E\left[1_{\{\xi_n=\eta_n=1\}}1_A\right]
 =E\left[1_{\{\xi_n=\eta_n=-1\}}1_A\right].
\end{equation}

Let $\mathcal C=\Big\{\{\Delta T_n=1\big\}\cap B, \{\Delta T_n=0\}\cap B\Big|B\in\mathcal H_{n-1}\Big\}$. Since
$$\{\Delta T_n=1\big\}=\{\xi_n=\eta_n=1\big\}\bigcup \{\xi_n=\eta_n=-1\},$$
 (\ref{condition=}) holds for every $A\in \mathcal C$. And it is easy to check that $\mathcal C$ is a $\pi-$system and $\sigma (\mathcal C)=\mathcal H_{n-1}\bigvee \mathcal G_n$.

Now define $\mathcal L=\left\{A\in \mathcal H_{n-1}\bigvee \mathcal G_n \left|(\ref{condition=})\ \mbox{holds for}\ A\right.\right\}$, then $\mathcal L$ is a $\lambda-$system containing $\mathcal C$. By \emph{Monotone Classes Theorem}, $\mathcal H_{n-1}\bigvee \mathcal G_n=\sigma (\mathcal C)\subseteq \mathcal L$. Accordingly, $\mathcal L=\mathcal H_{n-1}\bigvee \mathcal G_n$. Similar arguments apply to the second equality and the lemma follows.
\end{proof}

Now we finish the proof of Theorem \ref{XYT independent}

\begin{proofofof}{\ref{XYT independent}}
First note that the independence of $X$, $Y$ and $T$ is equivalent to
$$P(\Delta X_{n_1}=a_1,\cdots,\Delta X_{n_k}=a_k,\Delta Y_{m_1}=b_1,\cdots,\Delta Y_{m_l}=b_l|\mathcal{G}_{\infty})=\frac1{2^{k+l}},$$
for any $n_i,m_j \in \mathbb N$ and $a_i,b_j\in\{1,-1\},i=1,2,\cdots,k,j=1,2,\cdots,l$,
which can be rewritten as
\begin{equation}\label{equiv inde}
P\left(\left.\tilde\xi_{\alpha_{n_i}}=a_i,\tilde\xi_{\beta_{m_j}}=b_j;i=1,2,\cdots,k;j=1,2,\cdots,l
\right|\mathcal{G}_{\infty}\right)=\frac1{2^{k+l}}.
\end{equation}

For the "if" part, consider a possible value set for the $\alpha_{n_i}s$ and $\beta_{m_j}s$ in (\ref{equiv inde}):
$\{d_i,f_j\in \mathbb N, d_i\geq n_i,f_j\geq m_j,i=1,2,\cdots,k;j=1,2,\cdots,l\}$ and let $N=\max\{d_i,f_j,i=1,2,\cdots,k;j=1,2,\cdots,l\}$. Without loss of generality, we assume $d_1=N$. Because $N<\infty$, then
\begin{eqnarray}
&& P\left(\left.\tilde\xi_{\alpha_{n_i}}=a_i,\tilde\xi_{\beta_{m_j}}=b_j,
    \alpha_{n_i}=d_i, \beta_{m_j}=f_j,1\leq i\leq k, 1\leq j\leq l
\right|\mathcal{G}_{\infty}\right)\nonumber \\
&& =P\left(\left.\xi_{\alpha_{n_i}}=a_i,\xi_{\beta_{m_j}}=b_j,
    \alpha_{n_i}=d_i, \beta_{m_j}=f_j,1\leq i\leq k, 1\leq j\leq l
\right|\mathcal{G}_{\infty}\right)\nonumber \\
&& =P\left(\left.\xi_{d_i}=a_i,\xi_{f_j}=b_j,
    \alpha_{n_i}=d_i, \beta_{m_j}=f_j, 2\leq i\leq k, 1\leq j\leq l
; T_{N-1}=n_1-1, \xi_N=\eta_N=a_1
\right|\mathcal{G}_{\infty}\right)\nonumber\\
&& =1_{\{T_{N-1}=n_1-1\}}P\left(\left. \xi_{d_i}=a_i,\xi_{f_j}=b_j,
    \alpha_{n_i}=d_i, \beta_{m_j}=f_j, 2\leq i\leq k, 1\leq j\leq l; \xi_N=\eta_N=a_1
\right|\mathcal{G}_{N}\right)\nonumber\\
&&= 1_{\{T_{N-1}=n_1-1\}}E\left[\left.1_{\{\xi_{d_i}=a_i,\xi_{f_j}=b_j,
    \alpha_{n_i}=d_i, \beta_{m_j}=f_j, 2\leq i\leq k, 1\leq j\leq l\}}P\left(\xi_N=\eta_N=a_1
|\mathcal H_{N-1}\vee \mathcal{G}_{N}\right)\right|\mathcal{G}_{N}\right].\label{TN}
\end{eqnarray}
The third equality is from a equivalent condition of the condition C1), see \cite[Chapter 6]{bielecki2013credit}. 
Applying Lemma \ref{lemma1-2}, from \eqref{TN} we have that
\begin{eqnarray*}
&&P\left(\left.\xi_{\alpha_{n_i}}=a_i,\xi_{\beta_{m_j}}=b_j,
    \alpha_{n_i}=d_i, \beta_{m_j}=f_j,1\leq i\leq k, 1\leq j\leq l
\right|\mathcal{G}_{\infty}\right)\nonumber \\
&=&\frac12 1_{\{T_{N-1}=n_1-1\}}1_{\{\Delta T_{N}=1\}}P\left(\left.\xi_{d_i}=a_i,\xi_{f_j}=b_j,
    \alpha_{n_i}=d_i, \beta_{m_j}=f_j,2\leq i\leq k, 1\leq j\leq l\right|\mathcal{G}_{N}\right)\\
&=&\frac12 1_{\{\alpha_{n_1}=N\}}P\left(\left.\xi_{d_i}=a_i,\xi_{f_j}=b_j,
    \alpha_{n_i}=d_i, \beta_{m_j}=f_j,2\leq i\leq k, 1\leq j\leq l\right|\mathcal{G}_{N}\right).
\end{eqnarray*}
Notice that $\mathcal G_N$ and $\mathcal H_n$ are conditionally independent given $\mathcal G_n$ for any $n<N$. Then applying the same method to the second largest value of $\{d_i,f_j,i=1,\cdots,k;j=1,\cdots, l\}$, and then to the third, then the forth,etc. Finally, we get
\begin{eqnarray*}
&&P\left(\left.\tilde\xi_{\alpha_{n_i}}=a_i,\tilde\xi_{\beta_{m_j}}=b_j,
    \alpha_{n_i}=d_i, \beta_{m_j}=f_j;i=1,2,\cdots,k;j=1,2,\cdots,l
\right|\mathcal{G}_{\infty}\right)\\
&=&\frac{1}{2^{k+l}}1_{\{\alpha_{n_i}=d_i, \beta_{m_j}=f_j;1\leq i\leq k, 1\leq j\leq l\}}.
\end{eqnarray*}
Taking $\{d_i,f_j,i=1,2,\cdots,k;j=1,2,\cdots,l\}$ over all possible finite values and summing up the results, we have
\begin{eqnarray*}
&&P\left(\left.\tilde\xi_{\alpha_{n_i}}=a_i,\tilde\xi_{\beta_{m_j}}=b_j;
\alpha_{n_i}<\infty, \beta_{m_j}<\infty, 1\leq i\leq k, 1\leq j\leq l
\right|\mathcal{G}_{\infty}\right)\\
&=& \frac{1}{2^{k+l}}1_{\{\alpha_{n_i}<\infty, \beta_{m_j}<\infty,1\leq i\leq k, 1\leq j\leq l\}}.
\end{eqnarray*}
If some of $\{\alpha_{n_i}, \beta_{m_j},i=1,2,\cdots,k;j=1,2,\cdots,l\}$
are valued $\infty$, by independence of $\{\zeta_n,n\geq1\},\{\psi_n,n\geq1\}$ and $\mathcal{F}_{\infty}$, we still have similar results.
Thus (\ref{equiv inde}) holds and then $X$,$Y$ and $T$ are independent.

For the "only if" part, when $X,Y,T$ are mutually independent,
\begin{align*}&P(X_{T_1}\in C_1,Y_{S_1}\in D_1,\cdots,X_{T_n}\in C_n,Y_{S_n}\in D_n|\mathcal{G}_{\infty})\\
=&P(X_{T_1}\in C_1,Y_{S_1}\in D_1,\cdots,X_{T_n}\in C_n,Y_{S_n}\in D_n|T_1,\cdots,T_n)\\
=&P(X_{T_1}\in C_1,Y_{S_1}\in D_1,\cdots,X_{T_n}\in C_n,Y_{S_n}\in D_n|\mathcal{G}_n)
\end{align*}
for any $C_i, D_i\in \mathcal B(R)$. Let
$$\mathcal{P}=\big\{\{X_{T_1}\in C_1,Y_{S_1}\in D_1,\cdots,X_{T_n}\in C_n,Y_{S_n}\in D_n\}: C_i,D_i\in\mathcal{B}(R),1\le i\le n\big\},$$ and $\mathcal Q=\big\{A\in \mathcal H_n: P\left(A|\mathcal G_\infty\right)=P\left(A|\mathcal G_n\right)\big\},$ then $\mathcal{P}$ is a $\pi$-system, $\mathcal{Q}$ is a $\lambda$-system, $\mathcal P\subseteq\mathcal Q\subseteq \mathcal H_n$, and $\mathcal{H}_n=\sigma(\mathcal{P})$. In consequence, $\mathcal Q=\mathcal H_n$. And thus $\mathcal{H}_n$ and $\mathcal G_{\infty}$ are conditionally independent given $\mathcal{G}_n$.
\end{proofofof}

\begin{rem}
In fact, the condition C1) is equivalent to the condition C2):
\begin{description}
  \item[C2)] There exists sub filtrations of $\mathbb F$, $\mathcal G$ and $\mathcal H$, that $\sigma(T_k,k\le n)\subseteq\mathcal{G}_n$, $\sigma(B_k,W_k,k\le n)\subseteq\mathcal{H}_n$ and $\mathcal{H}_n$ and $\mathcal{G}_{\infty}$ are independent under the condition $\mathcal{G}_n$. Furthermore, $$P(\xi_n=\eta_n=1|\mathcal{G}_n\bigvee\mathcal{H}_{n-1})=P(\xi_n=\eta_n=-1|\mathcal{G}_n\bigvee\mathcal{H}_{n-1}), $$ $$P(\xi_n=-\eta_n=1|\mathcal{G}_n\bigvee\mathcal{H}_{n-1})=P(\xi_n=-\eta_n=-1|\mathcal{G}_n\bigvee\mathcal{H}_{n-1}).$$
\end{description}

In some cases, the condition C2) is much easier to check than the condition C1).
\end{rem}

\subsection{An Example from Finance}
In finance, when applying the Bachelier asset-price model $P_t=\sigma Z_t$ or the Black-Scholse-Merton model \cite{black1973pricing} (under risk-neutral probability) $dP_t/P_t=rdt+\sigma dZ_t$, where $Z$ stands for a standard Brownian motion, correlation between asset prices is described by correlation of Brownian motions. In either case, for $t>s,$ $\{Z_t-Z_s>0\}=\{P_{t}>P_{s}\}$, i.e., the asset price goes strictly up in the time interval $(s,t]$. Hence the decomposition focuses on common and counter movements of two discretized price sequences.

\begin{example}
Suppose $(Z_t^1,Z_t^2)$ is a 2-dimensional Brownian motion with correlation coefficient $\rho$. Given $0<t_1<\cdots<t_n<\cdots$, let
$\xi_n\triangleq1_{\{\Delta Z_{t_n}^1>0\}}-1_{\{\Delta Z_{t_n}^1\leq0\}}$, $\eta_n\triangleq1_{\{\Delta Z_{t_n}^2>0\}}-1_{\{\Delta Z_{t_n}^2\leq0\}}$ and $\mathcal F_n\triangleq\sigma(Z^1_{t_k},Z^2_{t_k},k\leq n)$. Then by basic properties of Brownian motion, we get
$$P(\xi_n=1|\mathcal F_{n-1})=P(\Delta Z_{t_n}^1>0|\mathcal F_{n-1})=P(\Delta Z_{t_n}^1>0)=\frac12.$$
Similarly, $$P(\xi_n=-1|\mathcal F_{n-1})=P(\eta_n=1|\mathcal F_{n-1})=P(\eta_n=-1|\mathcal F_{n-1})=\frac12.$$
In this case, $$P(\xi_n=1,\eta_n=1|\mathcal{F}_{n-1})
=P(\Delta Z_{t_n}^1>0,\Delta Z_{t_n}^2>0|\mathcal{F}_{n-1})=\Phi(0,0;\rho),$$
 where $\Phi(x,y;\rho)$ is the c.d.f of a standard 2-dimensional normal distribution with correlation coefficient $\rho$. Note that the condition C1) is satisfied, thus $X$, $Y$ and $T$ are independent random processes, with $P(\Delta X_n=1)=P(\Delta Y_n=1)=P(\Delta X_n=-1)=P(\Delta Y_n=-1)=\frac12$,$P(\Delta T_n=1)=2\Phi(0,0;\rho)$ and $P(\Delta T_n=0)=1-2\Phi(0,0;\rho)$.

 Consider a simple case that there are only these two assets in the market. Then $X$ reflects the trend of market or systematic risk, $Y$ reflects the specific risk and $T$ represents how much time the asset price goes up and down along the market trend. For example, if we deduce from the recent market data $X > 0,Y > 0,T > S$, then we may conclude that the recent market seems more likely in an increasing trend; and if $T<S$, it seems the first stock is more likely to increase.

\end{example}

\section{Conclusion}

In this paper, we characterize two correlated random walks $B$ and $W$ by $X,Y$ and $T$, where
 $X$ shows the common movements of $B$ and $W$; $Y$ shows its counter movements; and $T_n$ represents the number of common movements till step $n$.
 Under some conditions, we prove that $X$ and $Y$ are two independent random walks. Consequently, $T$ contains all the dependency structure information of $B$ and $W$. We also provide a sufficient and necessary condition for $X,Y$ and $T$ to be mutually independent. 

\section*{Acknowledgement}
Chen and Yang's research was supported by the National Natural Science Foundation of China
(Grants No. 11671021), and Cheng's research was supported by the National Natural Science Foundation of China
(Grants No. 11601018).

\bibliography{reference}

\end{document}